\documentclass[12pt]{amsart}
\usepackage{amssymb,amsmath}

\headsep=1cm \numberwithin{equation}{section}
\newtheorem{theorem}{Theorem}[section]
\newtheorem{lemma}[theorem]{Lemma}
\newtheorem{proposition}[theorem]{Proposition}
\newtheorem{corollary}[theorem]{Corollary}

\theoremstyle{definition}

\theoremstyle{remark}
\newtheorem{remark}[theorem]{Remark}
\newtheorem{example}[theorem]{Example}

\newtheorem{acknowledgement}{Acknowledgement}

\newcommand{\depth}{\operatorname{depth}}

\newcommand{\reg}{\operatorname{reg}}

\newcommand{\fm}{\frak{m}}

\newcommand{\fa}{\frak{a}}

\begin{document}
\dedicatory{Dedicated to Professor Tony J. Puthenpurakal}
\author[Mafi]{ Amir Mafi}
\title[On the computation of the Ratliff-Rush closure]
{On the computation of the Ratliff-Rush closure, associated graded ring and invariance of a length}

\address{A. Mafi, Department of Mathematics, University of Kurdistan, P.O. Box: 416, Sanandaj,
Iran.} \email{a\_mafi@ipm.ir}



\begin{abstract} Let $(R,\fm)$ be a Cohen-Macaulay local ring of positive dimension $d$ and infinite residue field. Let $I$ be an $\fm$-primary ideal of $R$ and $J$ be a minimal reduction of $I$. In this paper we show that if $\widetilde{I^k}=I^k$ and $J\cap I^n=JI^{n-1}$ for all $n\geq k+2$, then $\widetilde{I^n}=I^n$ for all $n\geq k$. As a consequence, we can deduce that if $r_J(I)=2$, then $\widetilde{I}=I$ if and only if $\widetilde{I^n}=I^n$ for all $n\geq 1$. Moreover, we recover some main results of [\ref{Cpv}] and [\ref{G}]. Finally, we give a counter example for Question 3 of [\ref{P1}].

\end{abstract}

\maketitle

\section{Introduction}
Throughout this paper, we assume that $(R,\fm)$ is a Cohen-Macaulay local ring of positive dimension $d$, infinite residue field and $I$ an $\fm$-primary ideal of $R$. An ideal $J\subseteq I$ is called a reduction of $I$ if $I^{n+1}=JI^n$ for some $n\in\mathbb{N}$. A reduction $J$ is called a minimal reduction of $I$ if it does not properly contain a reduction of $I$. The least such $n$ is called the reduction number of $I$ with respect to $J$, and denoted by $r_J(I)$. These notions were introduced by Northcott and Rees [\ref{Nr}], where they proved that minimal reductions of $I$ always exist if the residue field of $R$ is infinite. Recall that $x\in I$ is a superficial element of $I$ if there exists $k\in\mathbb{N}_0$ such that $I^{n+1}:x=I^n$ for all $n\geq k$. A set of elements $x_1,...,x_d$ is a superficial sequence of $I$ if $x_i$ is a superficial element of $I/{(x_1,...,x_{i-1})}$ for all $i=1,...,d$. A superficial sequence  $x_1,...,x_d$ of $I$ is called tame if  $x_i$ is a superficial element of $I$, for all $i=1,...,d$. Elias [\ref{E1}] defined and proved the tame superficial sequence exists (see also [\ref{Drt}]). Swanson [\ref{S}] proved that if $x_1,...,x_d$ is a superficial sequence of $I$, then $J=(x_1,...,x_d)$ is a minimal reduction of $I$. It is known that every minimal reduction can be generated by superficial sequence (see [\ref{Sa}] or [\ref{Drt}]).

The Ratliff-Rush closure of $I$ is defined as the ideal $$\widetilde{I}=\cup_{n\geq 1}(I^{n+1}:I^n).$$ It is a refinement of the integral closure of $I$ and $\widetilde{I}=I$ if $I$ is integrally closed (see [\ref{Rr}]). The Ratliff-Rush filtration $\widetilde{I^n}$, $n\in\mathbb{N}_0$, carries important information on the associated graded ring $G(I)=\bigoplus_{n\geq 0}{I^n}/{I^{n+1}}$. For example, Heinzer, Lantz and Shah [\ref{Hls}] showed that the $\depth G(I)\geq 1$  if and only if $\widetilde{I^n}=I^n$ for all $n\in\mathbb{N}_0$.
The aim of this paper is to compute the Ratliff-Rush closure in some senses and as an application, we shall reprove some main results of [\ref{Cpv}], [\ref{Gr}] and [\ref{G}]. Finally, we reprove Theorem 1 of [\ref{P1}] and Theorem 1.6 of [\ref{Ah}] with a much easier proof, and we also give a counter example for Question 3 of [\ref{P1}].  This example also says that Theorem 1.8 of [\ref{Ah}] does not hold in general. For any unexplained notation or terminology, we refer the reader to [\ref{Bh}] and [\ref{Hs}].

\section{ Ratliff-Rush closure, associated graded ring}

\begin{proposition} Let $d=2$, $x_1,x_2$ be a superficial sequence of $I$ and $J=(x_1,x_2)$. Let $k\in\mathbb{N}_0$ such that $J\cap I^n=JI^{n-1}$ for all $n\geq k+1$. Then $\widetilde{I^n}=I^n$ for all $n\geq 1$ if and only if $I^n:x_1=I^{n-1}$ for $n=1,...,k$.
\end{proposition}

\begin{proof} $(\Longrightarrow)$ immediately follows by [\ref{P}, Corollary 2.7].\\
 $(\Longleftarrow)$. By [\ref{P}, Corollary 2.7], it is enough for us to prove
 $I^n:x_1=I^{n-1}$ for all $n\geq k$. By using induction on $n$, it is enough to prove the result for $n=k+1$. For this, firstly we prove that $JI^k:x_1=I^k$. But this is an elementary fact that $JI^k:x_1=(x_1I^k+x_2I^k):x_1=I^k+(x_2I^k:x_1)$ and also $x_2I^k:x_1=x_2I^{k-1}$. Hence $JI^k:x_1=I^k$. Therefore, by our assumption, we have $(J\cap I^{k+1}):x_1=I^k$ and so  we have $I^{k+1}:x_1=I^k$, as desired.
\end{proof}

The following result immediately follows by Proposition 2.1.

\begin{corollary} Let $d=2$, $x_1,x_2$ be a superficial sequence of $I$ and $J=(x_1,x_2)$. Let $k\in\mathbb{N}_0$ such that $r_J(I)=k$. Then $\widetilde{I^n}=I^n$ for all $n\geq 1$ if and only if $I^n:x_1=I^{n-1}$ for $n=1,...,k$.
\end{corollary}

\begin{corollary} Let $d=2$, $x_1,x_2$ be a superficial sequence of $I$ and $J=(x_1,x_2)$ such that $r_J(I)=2$. Then $\widetilde{I^n}=I^n$ for all $n\geq 1$ if and only if $I^2:x_1=I$.
\end{corollary}

The Hilbert-Samuel function of $I$ is the numerical function that measures the growth of the length of $R/I^n$ for all $n\in\mathbb{N}$. For all $n$ large this function $\lambda(R/I^n)$ is a polynomial in $n$ of degree $d$ $$\lambda(R/I^n)=\sum_{i=0}^d(-1)^ie_i(I){{n+d-i-1}\choose{d-i}},$$ where $e_0(I),e_1(I),...,e_d(I)$ are called the Hilbert coefficients of $I$.
Let $A=\bigoplus_{m\geq 0}A_m$ be a Notherian graded ring where $A_0$ is an Artinian local ring, $A$ is generated by $A_1$ over $A_0$ and $A_{+}=\bigoplus_{m>0}A_m$. Let $H_{A_{+}}^i(A)$ denote the i-th local cohomology module of $A$ with respect to the graded ideal $A_+$ and set $a_i(A)=\max\{m\vert\ \  [H_{A_{+}}^i(A)]_m\neq 0\}$ with the convention $a_i(A)=-\infty$, if $H_{A_{+}}^i(A)=0$. The Castelnuovo-Mumford regularity is defined by $\reg (A):=\max\{a_i(A)+i\vert\ \ i\geq 0\}$

\begin{proposition}
 Let $d=2$ and $J$ be a minimal reduction of $I$ such that $r_J(I)=2$. If $\widetilde{I}=I$, then we have the following:
\begin{itemize}
\item[(i)]
 $\reg{ G(I)}=2$.
\item[(ii)] $e_2(I)=\lambda(I^2/{JI})$.
\end{itemize}
\end{proposition}

\begin{proof}
The case $(i)$ follows by Corollary 2.3 and [\ref{M}, Theorem 2.1 and Corollay 2.2] and the case $(ii)$ follows by Corollary 2.3 and [\ref{Cpr}, Theorem 3.1].

\end{proof}

\begin{remark}
 Let $d=2$, $\widetilde{I}=I$ and $J$ be a minimal reduction of $I$. If $\reg{ G(I)}=3$, then by [\ref{M}, Lemma 1.2 and Corollary 2.2], [\ref{T}, Proposition 3.2] and Proposition 2.4 we have $r_J(I)=3$.
\end{remark}

The following result is an improvement of [\ref{H}, Theorem 2.11] and [\ref{I}, Proposition 16].
\begin{proposition} Let $d=2$,  $\widetilde{I}=I$ and $J$ be a minimal reduction of $I$. Then $r_J(I)=2$ if and only if $P_I(n)=H_I(n)$ for $n=1,2$, where $H_I(n)$ and $P_I(n)$ are
the Hilbert-Samuel function and  the Hilbert-Samuel polynomial respectively.
\end{proposition}

\begin{proof} $(\Longrightarrow)$ let $r_J(I)=2$. Then by Corollary 2.3, $\widetilde{I^n}=I^n$ for all $n\geq 1$ and so by [\ref{I}, Proposition 16] we have $H_I(n)=P_I(n)$ for all $n=1,2$.\\
$(\Longleftarrow)$ is clear by [\ref{I}, Proposition 16].
\end{proof}

\begin{remark} Let $J$ be a minimal reduction of $I$, $x_1\in J$ and $\overline{I}=I/{(x_1)}$, $\overline{J}=J/{(x_1)}$. Then, by definition of reduction number, we have
\begin{itemize}\item[(i)]
If $r_{\overline{J}}(\overline{I})=k$ and $I^{k+1}:x_1=I^k$, then $r_J(I)=k$.
\item[(ii)] If $d=2$ and $I^2:x_1=I$, Then $r_{\overline{J}}(\overline{I})\leq 2$ if and only if $r_J(I)\leq 2$.
\end{itemize}
\end{remark}

\begin{lemma} Let $d=2$ and $J$ be a minimal reduction of $I$ such that $J\cap I^n=JI^{n-1}$ for $n=1,...,t$. If $r_{\overline{J}}(\overline{I})=k$ and $\lambda(I^{n+1}/{JI^n})=\lambda({\overline{I}}^{n+1}/{{\overline{J}}{\overline{I}}^n})$ for $n=t,...,k-1$. Then $I^{n+1}:x_1=I^n$ for $n=0,...,k-1$.
\end{lemma}

\begin{proof} By [\ref{E}, Proposition 1.7(ii)], $(x_1)\cap I^n=x_1I^{n-1}$ for $n=1,...,t$ and so $I^n:x_1=I^{n-1}$ for $n=1,...,t$. Now, consider the exact sequence $$0\longrightarrow{I^{n+1}:x_1}/{JI^n:x_1}\longrightarrow I^{n+1}/{JI^n}\longrightarrow{\overline{I}}^{n+1}/{{\overline{J}}{\overline{I}}^n}\longrightarrow 0.\ \ (\dagger)$$
By our assumption, $I^{n+1}:x_1=JI^n:x_1$ for $n=t,...,k-1$. Assume that $yx_1\in JI^{t}$. Then we have $yx_1=\alpha_1x_1+\alpha_2x_2$ for some $\alpha_1,\alpha_2\in I^t$. Hence $(y-\alpha_1)x_1=\alpha_2x_2\in x_2I^t$ and since $x_1,x_2$ is a regular sequence, we obtain $y-\alpha_1=sx_2$ for some $s\in R$. Since $(y-\alpha_1)x_1=sx_1x_2\in x_2I^t$ and $x_2$ is a non-zerodivisor, it follows that $sx_1\in I^t$ and so $s\in I^t:x_1$. Therefore $s\in I^{t-1}$ and so $y\in I^t$. Thus by repeating this argument, we obtain $I^{n+1}:x_1=I^n$ for $n=0,...,k-1$, as desired.
\end{proof}

The following result was proved in [\ref{Hu}, Theorem 2.4], [\ref{Cpv}, Theorem 3.10] and [\ref{R}, Theorem 3.7], and we give a simplified proof.
\begin{proposition} Let $J$ be a minimal reduction of $I$ such that $J\cap I^n=JI^{n-1}$ for $n=1,...,t$ and $\lambda(I^{t+1}/{JI^t})\leq 1$. Then
$\depth G(I)\geq{d-1}$.

\end{proposition}

\begin{proof} By using Sally's descent, we may deduce the problem to the case of $d=2$. Set $r_{\overline{J}}(\overline{I})=k$. Then, by using the exact sequence $(\dagger)$, we have $\lambda({\overline{I}}^{n+1}/{{\overline{J}}{\overline{I}}^n})=\lambda(I^{n+1}/{JI^n})\leq 1$ for $n=t,...,k-1$. By Lemma 2.8, we have $I^{n+1}:x_1=I^n$ for $n=0,...,k-1$. By [\ref{Hu}, Proposition 1.1], we know that
$\sum_{n\geq 0}\lambda(\widetilde{I^{n+1}}/{J\widetilde{I^n}})=e_1(I)=e_1(\overline{I})=\sum_{n=0}^{k-1}
\lambda(I^{n+1}/{JI^n})=\sum_{n=0}^{t-1}\lambda(I^{n+1}/{JI^n})+k-t$. Therefore by [\ref{R1}, Theorem 1.3], we have $r_J(I)\leq k$. Thus by Lemma 2.8 and Corollary 2.2, we obtain $\widetilde{I^n}=I^n$ for all $n\geq 1$. Hence $\depth G(I)\geq 1$, as required.

\end{proof}

\begin{lemma} Let $d=2$ and $J=(x_1,x_2)$ a minimal reduction of $I$ such that $J\cap I^n=JI^{n-1}$ for all $n\geq 3$. If either $I^2:x_1=I$ or $I^2:x_2=I$, then $\widetilde{I^n}=I^n$ for all $n\geq 1$. In particular $\depth G(I)\geq 1$.
\end{lemma}
\begin{proof}
By using the same argument that was used in the proof of proposition 2.1, the result immediately follows.

\end{proof}

\begin{lemma}
 Let $d=2$ and $J=(x_1,x_2)$ be a minimal reduction of $I$ such that $\lambda(J\cap I^2 /JI)\leq 1$. Then either $I^2:x_1=I$ or $I^2:x_2=I$.
\end{lemma}

\begin{proof}
If $\lambda(J\cap I^2/{JI+I^2\cap(x_1)})=1$, then $I^2\cap(x_1)\subseteq JI$ and so $I^2\cap(x_1)\subseteq[x_1I+x_2I]\cap(x_1)$. Therefore $I^2\cap(x_1)={x_1}I$ and so $I^2:x_1=I$. If $\lambda(J\cap I^2/{JI+I^2\cap(x_1)})=0$, then $I^2\cap(x_1)+Ix_2=J\cap I^2$. Hence $I^2\cap(x_1x_2)+Ix_2=I^2\cap(x_2)$ and so $Ix_2=I^2\cap(x_2)$. Thus $I^2:x_2=I$.
\end{proof}

The following result was proved in [\ref{G}, Theorem 3.2] and [\ref{Gr}, Corollary 1.5] and we give an easier proof

\begin{proposition}
Let $J$ be a minimal reduction of $I$ such that $J\cap I^n=JI^{n-1}$ for all $n\geq 3$. If $\lambda(J\cap I^2/IJ)\leq 1$, then $\depth G(I)\geq{d-1}$.
\end{proposition}

\begin{proof}
By Sally's descent, we may assume that $d=2$. Now, by using Lemmas 2.11 and 2.10 the result follows.
\end{proof}

\begin{theorem} Let $d\geq 3$ and $k\in\mathbb{N}_0$ such that $\widetilde{I^k}=I^k$. If $x_1,...,x_d$ is a tame superficial sequence of $I$ and $J=(x_1,...,x_d)$ such that $J\cap I^n=JI^{n-1}$ for all $n\geq{k+2}$, then ${\fa}^mI^n:x_1={\fa}^mI^{n-1}$ for all $n\geq{k+1}$ and all $m\in\mathbb{N}_0$, where $\fa=(x_2,...,x_d)$. In particular, $\widetilde{I^n}=I^n$ for all $n\geq k$.
\end{theorem}

\begin{proof} We will proceed by induction on $n$. Assume $n=k+1$. Then by [\ref{Ma}, Lemma 2.7] and our assumption we have ${\fa}^mI^{k+1}:x_1\subseteq{\fa}^m\widetilde{I^{k+1}}:x_1={\fa}^m\widetilde{I^k}=
{\fa}^mI^k$. Therefore ${\fa}^mI^{k+1}:x_1={\fa}^mI^k$ for all $m\in\mathbb{N}_0$. Assume $n\geq{k+1}$ and that for all $t$ with ${k+1}\leq t\leq n$ and all $m\in\mathbb{N}_0$, ${\fa}^mI^t:x_1={\fa}^mI^{t-1}$. We show that for all $m\in\mathbb{N}_0$, ${\fa}^mI^{n+1}:x_1={\fa}^mI^n$. Let $yx_1$ be an element of ${\fa}^mI^{n+1}$. Then $yx_1\in{\fa}^m$ and by using [\ref{Ma}, Lemma 2.1] we obtain $y\in{\fa}^m$. Therefore we can write the expression, $y=\sum_{i_2+...+i_d=m}r_{i_2...i_d}x_2^{i_2}...x_d^{i_d}$. Since the element $yx_1$ belongs to ${\fa}^mI^{n+1}$ too, we obtain the following equalities $$\sum_{i_2+...+i_d=m}r_{i_2...i_d}x_1x_2^{i_2}...x_d^{i_d}=yx_1=\sum_{i_2+...+i_d=m}s_{i_2...i_d}x_2^{i_2}...x_d^{i_d},$$ where $s_{i_2...i_d}\in I^{n+1}$ for all $i_2,...,i_d$ such that $i_2+...+i_d=m$. As $x_1,...,x_d$ is a regular sequence in $R$, by equating coefficients in the previous expressions, we get $r_{i_2...i_d}x_1-s_{i_2...i_d}\in(x_2,...,x_d)$ for all $i_2,...,i_d$ such that $i_2+...+i_d=m$. Hence $s_{i_2...i_d}\in J\cap I^{n+1}$ and by our assumption we obtain $s_{i_2...i_d}\in JI^n$ for all  $i_2,...,i_d$ such that $i_2+...+i_d=m$. Hence, going back to the equalities we wrote for $yx_1$, we obtain $yx_1\in{\fa}^mJI^n={\fa}^{m+1}I^n+x_1{\fa}^mI^n$. Therefore we have $${\fa}^mI^{n+1}\cap (x_1)\subseteq{\fa}^{m+1}I^n\cap{(x_1)}+x_1{\fa}^mI^n=x_1({\fa}^{m+1}I^n:x_1)+x_1{\fa}^mI^n.$$
 By applying the inductive hypothesis we get ${\fa}^mI^{n+1}\cap{(x_1)}\subseteq x_1{\fa}^{m+1}I^{n-1}+x_1
 {\fa}^mI^n=x_1{\fa}^mI^n$. This proves that ${\fa}^mI^{n+1}:x_1\subseteq{\fa}^mI^n$ and so ${\fa}^mI^{n+1}:x_1={\fa}^mI^n$ for all $m\in\mathbb{N}_0$. In particular, if we set $m=0$, then $I^{n+1}:x_1=I^n$ for all $n>k$ and so by [\ref{P}, Corollary 2.7], $\widetilde{I^n}=I^n$ for all $n\geq k$, as desired.

\end{proof}
The following result easily follows by Theorem 2.13.
\begin{corollary} Let $x_1,...,x_d$ be a tame superficial sequence of $I$ and $J=(x_1,...,x_d)$.
\begin{itemize}
\item[(i)] If $\widetilde{I}=I$ and $J\cap I^{n}=JI^{n-1}$ for all $n\geq 3$, then $\widetilde{I^n}=I^n$ for all $n\geq 1$. In particular $\depth G(I)\geq 1$.
\item[(ii)]If $r_J(I)=2$, then $\widetilde{I}=I$ if and only if $\depth G(I)\geq 1$.
\item[(iii)]
Let $k\in\mathbb{N}_0$ such that $r_J(I)=k+1$ and $\widetilde{I^k}=I^k$. Then $\widetilde{I^n}=I^n$ for all $n\geq k$.

\end{itemize}
\end{corollary}

The following example shows that the equality of Corollary 2.14(ii) maybe happen.
\begin{example}
Let $K$ be a field, $R=K[\![ x,y]\!]$, $I=(x^6,x^4y^2,x^3y^3,x^2y^4,xy^5,y^6)$ and $J=(x^6,y^6+x^4y^2)$. Then $r_J(I)=2$, $\depth G(I)=1$ and so $G(I)$ is not C.M.
\end{example}
\section{Invariance of a length}
Let $J=(x_1,...,x_d)$ be a minimal reduction of $I$. In [\ref{W}] Wang defined the following exact sequence for all $n,k$  $$0\longrightarrow T_{n,k}\longrightarrow \oplus^{{k+d-1}\choose{d-1}} I^n/{JI^{n-1}}\overset{\phi_k}{\longrightarrow} J^kI^n/{J^{k+1}I^{n-1}}
\longrightarrow 0,\ \ \ (*)$$
where $\phi_k=(x_1^k,x_1^{k-1}x_2,...,x_1^{k-1}x_d,...,x_d^k)$ and $T_{n,k}=\ker(\phi_k)$. He also showed that $T_{1,k}=0$ for all $k$ and if $d=1$, then $T_{n,k}=0$ for all $n,k$. By using the exact sequence $(*)$, we drive the following easy lemma and we leave the proof to the reader.
\begin{lemma} Let $t\in\mathbb{N}_0$ and $J=(x_1,...,x_d)$ be a minimal reduction of $I$. Then we have the following:
\begin{itemize}
\item[(i)] If $J\cap I^n=JI^{n-1}$ for $n=1,...,t$, then $T_{n,k}=0$ for $n=1,...,t$ and all $k$.
\item[(ii)] If $I$ is integrally closed, then $T_{2,k}=0$ for all $k$. In particular, if $I=m$, then $T_{2,k}=0$ for all $k$.
\end{itemize}
\end{lemma}

The following lemma is known see the proof of [\ref{Cpv}, Proposition 2.1].

\begin{lemma} Let $J=(x_1,...,x_d)$ be a minimal reduction of $I$. Then

$\lambda(I/J)=e_0(I)-\lambda(R/I)$ and $\lambda(I^{n+1}/{J^{n}I})=e_0(I){{n+d-1}\choose{d}}+\lambda(R/I){{n+d-1}\choose{d-1}}-\lambda(R/I^{n+1})$ for $n\geq 1$  which are independent of $J$.
\end{lemma}

In [\ref{P1}], Puthenpurakal proved that $\lambda({\fm}^3/{J{\fm}^2})$ is independent of the minimal reduction $J$ of $\fm$ and subsequently Ananthnarayan and Huneke [\ref{Ah}] extend it for $n$-standard admissible $I$-filtrations.\\
The following result was proved in [\ref{P1}, Theorem 1] and [\ref{Ah}, Theorem 3.5]. We reprove it with a much easier proof.
\begin{theorem}
Let $t\in\mathbb{N}_0$ and $J=(x_1,...,x_d)$ be a minimal reduction of $I$. If $J\cap I^n=JI^{n-1}$ for $n=1,...,t$, then $\lambda(I^{n+1}/{JI^n})$ is independent of $J$ for $n=1,...,t$.
\end{theorem}

\begin{proof}
We have $\lambda(I^{n+1}/{JI^{n}})=\lambda(I^{n+1}/{J^{n}I})-\sum_{k=1}^{n-1}
\lambda(J^kI^{n+1-k}/{J^{k+1}I^{n-k}})$. Therefore by Lemma 3.1 and the exact sequence $(*)$, we obtain $\lambda(I^{n+1}/{JI^{n}})=\lambda(I^{n+1}/{J^{n}I})-\sum_{k=1}^{n-1}{{k+d-1}\choose{d-1}}\lambda(I^{n+1-k}/{JI^{n-k}})$. Now by using Lemma 3.2 and using induction on $n$, the result follows.

\end{proof}

The following example is a counterexample for Question 3 of [\ref{P1}] and it also says that Theorem 1.8 of [\ref{Ah}] does not hold in general.
The computations are performed by using Macaulay2 [\ref{Gs}], CoCoA [\ref{Ab}] and Singular [\ref{Glp}].

\begin{example} Let $K$ be a field and $S=K[\![ x,y,z,u,v ]\!]$, where $I=(x^2+y^5,xy+u^4,xz+v^3)$. Then $R=S/I$ is a Cohen-Macaulay local ring of dimension two, ideals $J_1=(y,z)R$ and $J_2=(z,u)R$ are minimal reduction of $\fm=(x,y,z,u,v)R$ and $\lambda({\fm}^4/{J_1{\fm}^3})=17$, $\lambda({\fm}^4/{J_2{\fm}^3})=20$.
\end{example}

\begin{acknowledgement} This paper was done while I was visiting the University of Osnabruck. I would like to thank the Institute of Mathematics of the University of Osnabruck for hospitality and
partially financial support and I also would like to express my very great appreciation to Professor Winfried Bruns for his valuable and constructive suggestions during the planning and development of this research work. Moreover, I would like to thank deeply grateful to the referee for the careful reading of the manuscript and the helpful suggestions.
\end{acknowledgement}


\end{document}